\documentclass[12pt]{article}

\usepackage{multirow}

\usepackage{amsthm}
\usepackage{amsmath}
\usepackage{amssymb}
\usepackage{bm}
\usepackage{mathrsfs}
\usepackage[dvips]{graphicx}

\usepackage{algorithmic}
\usepackage{algorithm}

\usepackage{color}

\usepackage{url}

\usepackage{supertabular}

\newtheorem{thm}{Theorem}

\newtheorem{lemma}[thm]{Lemma}

\theoremstyle{definition}

\theoremstyle{remark}

\numberwithin{thm}{section}

\DeclareMathAlphabet{\mathsfsl}{OT1}{cmss}{m}{sl}

\renewcommand{\phi}{\varphi}

\newcommand{\argmin}{\operatorname*{arg\; min}}

\newcommand{\bx}{\mathbf{x}}

\newcommand{\ba}{\mathbf{a}}
\newcommand{\bb}{\mathbf{b}}
\newcommand{\by}{\mathbf{y}}

\newcommand{\bT}{\mathbf{T}}
\newcommand{\bc}{\mathbf{c}}
\newcommand{\bz}{\mathbf{z}}

\newcommand{\bD}{\mathbf{D}}

\newcommand{\bw}{\mathbf{w}}

\def\reals{\mathbb{R}}
\def\calG{\mathcal{G}}
\def\calE{\mathcal{E}}
\def\calV{\mathcal{V}}
\def\bx{\mathbf{x}}

\def\bu{\mathbf{u}}

\def\b0{\mathbf{0}}

\def\bC{\mathbf{C}}
\def\bv{\mathbf{v}}
\def\bt{\mathbf{t}}
\def\bA{\mathbf{A}}

\def\bB{\mathbf{B}}

\def\prox{\mathrm{prox}}

\def\bI{\mathbf{I}}

\usepackage{amsmath}
\usepackage{graphicx}
\usepackage{natbib}
\usepackage{url} 

\newcommand{\blind}{0}

\begin{document}

\def\spacingset#1{\renewcommand{\baselinestretch}%
{#1}\small\normalsize} \spacingset{1}


\if0\blind
{
  \title{\bf An Algorithm for Graph-Fused Lasso Based on Graph Decomposition}
  \author{Feng Yu\hspace{.2cm}\\
    Department of Mathematics, University of Central Florida\hspace{.2cm}\\
    and\\
    Yi Yang \\
    Department of Mathematics and Statistics, McGill University\hspace{.2cm}\\
    and\\
    Teng Zhang\thanks{
    The authors are gratefully supported by National Science Foundation (NSF), grant CNS-1739736.} \\
    Department of Mathematics, University of Central Florida
    }
  \maketitle
} \fi

\if1\blind
{
  \bigskip
  \bigskip
  \bigskip
  \begin{center}
    {\LARGE\bf Title}
\end{center}
  \medskip
} \fi

\bigskip
\begin{abstract}
This work proposes a new algorithm for solving the graph-fused lasso (GFL), a method for parameter estimation that operates under the assumption that the signal tends to be locally constant over a predefined graph structure. The proposed method applies the alternating direction method of multipliers  (ADMM) algorithm and is based on the decomposition of the objective function into two components. While ADMM has been widely used in this problem, existing works such as network lasso decompose the objective function into the loss function component and the total variation penalty component. In comparison, this work proposes to decompose the objective function into two components, where one component is the loss function plus part of the total variation penalty, and the other component is the remaining total variation penalty. Compared with the network lasso algorithm, this method has a smaller computational cost per iteration and converges faster in most simulations numerically.
\end{abstract}

\noindent%
\vfill

\newpage
\spacingset{1.5} 
\section{Introduction}
In this article, we consider graph-fused lasso, an estimation method based on noisy observations and the assumption that the signal tends to be locally constant over a predefined graph structure. Given a graph $\calG=(\calV,\calE)$, where $\calV$ is the set of vertices and $\calE$ is the set of edges, we let $\bx_i\in\reals^p$ be the signal that  is associated with the $i$-the vertex of the graph,  then GFL is defined as the solution to the following optimization problem:
\begin{equation}\label{eq:problem0}
\{\hat{\bx}_i\}_{i\in\calV}=\argmin_{\{{\bx}_i\}_{i\in\calG}\subset\reals^p} \sum_{i\in\calV}f_i(\bx_i)+\lambda \sum_{(s,t)\in\calE}\|\bx_r-\bx_s\|,
\end{equation}
where the first component  is a loss function for the observation $\bx_i$, and the second component uses the total variation norm to penalize the difference between the two signals on the edges in the graph.

There have been extensive studies on \eqref{eq:problem0} with $p=1$ (i.e., $\bx_i$ are scalars) and many algorithms have been developed. When graph $\calG$ is a one-dimensional chain graph, then it is the standard fused lasso problem \citep{Tibshirani05sparsityand}. For this problem, there exist finite-step algorithms with computational costs of $O(n)$: a taut-string method proposed by~\cite{davies2001}, a method based on analyzing its dual problem by \cite{6579659}, a dynamic programming-based approach by~\cite{Johnson2013}, and a modular proximal optimization approach by~\cite{barberoTV14} all solve the problem with $O(n)$ complexity. When the $\calG$ is a two-dimensional grid graph, it has important applications in image denoising and it often referred to as total-variation denoising~\citep{RUDIN1992259}, and parametric max-flow algorithm \citep{Antonin2009} can be used to solve \eqref{eq:problem0} in finite steps. When the graph is a tree, Kolmogorov et al.~\citep{doi:10.1137/15M1010257} extended the dynamic programming approach of Johnson to solve the fused lasso problem. While these algorithms can find the exact solution in finite steps, they only apply to some specific graph structure and do not work for general graphs. In addition, they cannot be naturally generalized to the setting of $p>1$, which is sometimes called group fused lasso~\citep{bleakley:hal-00602121}.

In addition, many iterative algorithms based on  convex optimization algorithms have been proposed to solve \eqref{eq:problem0}. For example, \citet{Liu2010} use a projected gradient descent method
to solve the dual of \eqref{eq:problem0}, and reformulate it as the problem of finding an ``appropriate'' subgradient of the fused penalty at the minimizer. \citet{chen2012} propose the smoothing proximal gradient (SPG) method. ~\citet{Lin2014} proposed an alternating linearization
method. \citet{Yu2014} proposed a majorization-minimization
method. One of the more popular methods is the alternating direction method of multipliers (ADMM), due to its simplicity and competitive empirical performance. \citet{Ye2011} and~\citet{Wahlberg2012} proposed algorithms based on the
ADMM method. However, there is a step of solving a linear system for the $n\times n$ matrix $\bI+\rho \bD^T\bD$, which is usually in the order of $O(n^2)$. \citet{Zhu2015} proposed a modified ADMM algorithm that has a smaller computational cost of $O(n)$ in an update step, but this modification generally converges slower. \citet{Ramdas2015} proposed a special ADMM algorithm that used dynamic programming in one of the update step, which can be used in the trend filtering problem, or when $\bD$ has a diagonal structure. \citet{barberoTV14} propose a method based on the Douglas-Rachford decomposition for the two-dimensional grid graph, which can be considered as the dual algorithm of ADMM~\citep{Eckstein1992}. \citet{Tansey2015} leveraged fast 1D fused lasso solvers in an ADMM method based on graph decomposition, but it can only be applied to the case when $p=1$.  \citet{Hallac2015} proposed the network lasso algorithm based on ADMM that can be applied to any generic graph and any $p\geq 1$.

There exist other types of algorithms as well. \citet{friedman2007} and  \citet{Arnold2016} gave solution path algorithms (tracing out the solution over all $\lambda\geq 0$).  Some other algorithms include a working-set/greedy algorithm~\citep{doi:10.1137/17M1113436} and an algorithm based on an active set search~\citep{doi:10.1198/jcgs.2011.09203}.

Among all algorithms, the network lasso \citep{Hallac2015} is particularly interesting since it is scalable to any large graphs and can be applied to the case $p\geq 1$. The algorithm proposed in this work follows this direction and can be considered as an improvement of the network lasso algorithm. The main contribution of this work is a novel ADMM method by dividing the objective function into two parts based on graph decomposition so that one of the subgraphs does not contain any two adjacent edges. This method can be applied to any graph and can be generalized to some other problems such as trend filtering. Compared with the  network lasso algorithm in \citep{Hallac2015}, it reduces the computational complexity per iteration and achieves a faster convergence rate. 

The rest of this paper is organized as follows. In Section 2 we introduce our proposed method and analyze its computational complexity per iteration as well as convergence rates and establish the advantage of the proposed algorithm theoretically. Then we compare our algorithm with the network lasso algorithm in Section 3, both in simulated data sets and a real-life data set, which verifies the advantage of the proposed algorithm numerically.


\section{Proposed Method}
In this section, we will  review the network lasso algorithm \citep{Hallac2015} for solving \eqref{eq:problem0} in Section~\ref{sec:review},  present our algorithm in Sections~\ref{sec:proposed1} and \ref{sec:proposed2}, and analyze its performance in terms of computational cost per iteration and convergence rate in Sections~\ref{sec:implementation} and~\ref{sec:convergence}.

\subsection{Review: Network lasso}\label{sec:review}
\citet{Hallac2015} introduce the following ``Network Lasso'' algorithm: for any $(s,t)\in\calE$, introduce a pair of variables $\bz_{st}, \bz_{ts}\in\reals^p$, which are the copies of $\bx_r$ and $\bx_s$ respectively, and rewrite the problem \eqref{eq:problem0} as follows:
\begin{equation}\label{eq:networklasso}
\argmin_{\{{\bx}_i\}_{i\in\calG}, \{{\bz}_{st},{\bz}_{ts}\}_{(s,t)\in\calE}} \sum_{i\in\calV}f_i(\bx_i)+\lambda \sum_{(s,t)\in\calE}\|{\bz}_{st}-{\bz}_{ts}\|,\,\,\text{s.t. $\bx_s=\bz_{st}$ and $\bx_t=\bz_{ts}$ for all $(s,t)\in\calE$}.
\end{equation}
Then the standard ADMM routine would apply: let $\bu_{st}, \bu_{ts}$ be the dual variables for $\bx_s-\bz_{st}$ and $\bx_t-\bz_{ts}$ respectively, then the augmented Lagrangian is (here $x$ and $z$ represents $\{\bx_i\}_{i\in\calG}$ and $\{\bz_i\}_{i\in\calG}$):
\begin{align}\label{eq:lagrangian0}
L_\rho(x,z,u)=&\sum_{i\in\calV}f_i(\bx_i)+\sum_{(s,t)\in\calE}\Big(\lambda \|{\bz}_{st}-{\bz}_{ts}\|+\bu_{st}^T(\bx_s-\bz_{st})+\bu_{ts}^T(\bx_t-\bz_{ts})\\&+\frac{\rho}{2}\|\bx_s-\bz_{st}\|^2+\frac{\rho}{2}\|\bx_t-\bz_{ts}\|^2\Big)\nonumber
\end{align}
and the algorithm can be written as
\begin{align}\label{eq:ADMM}
&x^{(k+1)}=\argmin_{x}L_\rho(x,z^{(k)},u^{(k)})\\
&z^{(k+1)}=\argmin_{z}L_\rho(x^{(k+1)},z,u^{(k)})\\
&\bu_{st}^{(k+1)}=\bu_{st}^{(k)}+\rho(\bx_s^{(k+1)}-\bz_{st}^{(k+1)}),\,\,\bu_{ts}^{(k+1)}=\bu_{ts}^{(k)}+\rho(\bx_t^{(k+1)}-\bz_{ts}^{(k+1)}).
\end{align}
The advantage of this algorithm is that, in each iteration, the optimization problem can be decomposed into smaller subproblems: the updates of $x$ requires solving  problems in the form of $\min_{\bx_i} f_i(\bx_i)+\|\bx_i-\bt\|^2$, which has explicit solutions for a large range of $f_i$; and the updates of $z$ requires solving  $\min_{
\bz_{st}, \bz_{ts}} \|\bz_{ts}-\bt_1\|^2+\|\bz_{st}-\bt_2\|^2+\lambda\|\bz_{ts}-\bz_{st}\|$, which has explicit solutions.

\subsection{Proposed Method: A different way of splitting the objective function}\label{sec:proposed1}
In this section, we will propose another ADMM algorithm for solving \eqref{eq:problem0}, based on the reformulation as follows: We will divide the set of edges $\calE$ into $\calE_0$ and $\calE_1$, such that the set $\calE_0$ does not contain two neighboring edges, and then solve the following optimization problem:
\begin{align}\label{eq:proposed1}
\argmin_{\{\bx_i\}_{i\in\calV}, \{\bz_{st}\}_{(s,t)\in\calE_1}} &\left(\sum_{i\in\calV}f_i(\bx_i)+\lambda\sum_{(s,t)\in\calE_0}\|\bx_s-\bx_t\|\right) + \lambda\sum_{(s,t)\in\calE_1}\|\bz_{st}-\bz_{ts}\|,\\&\text{s.t. $\bx_s=\bz_{st}$ and $\bx_t=\bz_{ts}$ for all $(s,t)\in\calE_1$.}\nonumber
\end{align}
Since this formulation is different than \eqref{eq:networklasso}, its associated ADMM routine is also different. Let $\bu_{st}, \bu_{ts}$ be the dual variables for $\bx_s-\bz_{st}$ and $\bx_t-\bz_{ts}$ respectively, then the augmented Lagrangian is
\begin{align}\label{eq:lagrangian1}
\hat{L}_\rho(x,z,u)=\sum_{i\in\calV}f_i(\bx_i)&+\lambda\sum_{(s,t)\in\calE_0}\|\bx_s-\bx_t\|+\sum_{(s,t)\in\calE_1}\Big(\lambda \|{\bz}_{st}-{\bz}_{ts}\|+\bu_{st}^T(\bx_s-\bz_{st})\\&+\bu_{ts}^T(\bx_t-\bz_{ts})+\frac{\rho}{2}\|\bx_s-\bz_{st}\|^2+\frac{\rho}{2}\|\bx_t-\bz_{ts}\|^2\Big).
\end{align}
and the update formula is
\begin{align}\label{eq:ADMM1}
&x^{(k+1)}=\argmin_{x}\hat{L}_\rho(x,z^{(k)},u^{(k)})\\
&z^{(k+1)}=\argmin_{z}\hat{L}_\rho(x^{(k+1)},z,u^{(k)})\\
&\bu_{st}^{(k+1)}=\bu_{st}^{(k)}+\rho(\bx_s^{(k+1)}-\bz_{st}^{(k+1)}),\,\,\bu_{ts}^{(k+1)}=\bu_{ts}^{(k)}+\rho(\bx_t^{(k+1)}-\bz_{ts}^{(k+1)}).\label{eq:ADMM13}
\end{align}
While the update formula for $x$ \eqref{eq:ADMM1} is similar to \eqref{eq:ADMM}, it requires solving a slightly different problem  due to the additional component $\lambda\sum_{(s,t)\in\calE_0}\|\bx_s-\bx_t\|$. For any $(s,t)\in\calE_0$, the ADMM procedure needs to solve
\begin{equation}\label{eq:simplified2}
\argmin_{\bx_s,\bx_t\in\reals^p}f_s(\bx_s)+f_t(\bx_t)+\frac{\rho}{2}d_s(\bx_{s}-\bt_1)^2+\frac{\rho}{2}d_t(\bx_{t}-\bt_2)^2 + \lambda \|\bx_{s}-\bx_{t}\|,
\end{equation}
where $\bt_1,\bt_2\in\reals^p$ and $d_s$ denotes the degree of the vertex $s$ in the graph $(\calV, \calE_1)$. For many choices of $f_s$ and $f_t$ (for example, square functions), this problem has an explicit solution.

Intuitively, we expect that the proposed algorithm would achieve a faster convergence rate than  \eqref{eq:networklasso}: \eqref{eq:proposed1} has fewer ``dummy variables'' in the form of $\bz_{st}$ ($2|\calE_1|$ instead of $2|\calE|$) and \eqref{eq:lagrangian1} has fewer dual parameters than \eqref{eq:proposed1}. As a result, the Lagrangian in \eqref{eq:lagrangian1} contains fewer variables than \eqref{eq:lagrangian0} and we expect the algorithm to converge faster.
\subsection{An equivalent formulation}\label{sec:proposed2}
In this section, we propose an equivalent form of the update formula  \eqref{eq:ADMM1}-\eqref{eq:ADMM13}, with a smaller computational cost per iteration. In particular, we consider the  ADMM algorithm for solving
\begin{align}\label{eq:proposed3}
\argmin_{\{\bx_i\}_{i\in\calV}, \{\bz_{st}\}_{(s,t)\in\calE_1}} \left(\sum_{i\in\calV}f_i(\bx_i)+\lambda\sum_{(s,t)\in\calE_0}\|\bx_s-\bx_t\|\right) + \lambda\sum_{(s,t)\in\calE_1}\|\bz_{st}\|,\\\,\,\text{s.t. $\bz_{st}=\bx_s-\bx_t$.}\nonumber
\end{align}
For its Lagrangian
\begin{align}\label{eq:lagrangian3}
\tilde{L}_\rho(x,z,u)=\sum_{i\in\calV}f_i(\bx_i)&+\lambda\sum_{(s,t)\in\calE_0}\|\bx_s-\bx_t\|+\sum_{(s,t)\in\calE_1}\Big(\lambda \|{\bz}_{st}\|+\bu_{st}^T(\bz_{st}-\bx_s+\bx_t)\\&+\frac{\rho}{2}\|\bz_{st}-\bx_s+\bx_t\|^2\Big),
\end{align}
the preconditioned ADMM algorithm can be written as
\begin{align}\label{eq:ADMM3}
&x^{(k+1)}=\argmin_{x}\tilde{L}_\rho(x,z^{(k)},u^{(k)})+\frac{\rho}{2}\sum_{(s,t)\in\calE_1}\|\bx_s+\bx_t-\bx_s^{(k)}-\bx_t^{(k)}\|^2\\
&z^{(k+1)}=\argmin_{z}\tilde{L}_\rho(x^{(k+1)},z,u^{(k)})\\
&u_{st}^{(k+1)}=u_{st}^{(k)}+\rho(z_{st}^{(k+1)}-x_s^{(k+1)}+x_t^{(k+1)}).\label{eq:ADMM33}
\end{align}
It is called preconditioned ADMM due to the additional component $\frac{\rho}{2}\sum_{(s,t)\in\calE_1}\|\bx_s+\bx_t-\bx_s^{(k)}-\bx_t^{(k)}\|^2$ in the update of $x$.

Compared with  the update formula  \eqref{eq:ADMM1}-\eqref{eq:ADMM13},  the computational cost of \eqref{eq:ADMM3}-\eqref{eq:ADMM33} is smaller since there are no ``dummy variables'' $\bz_{ts}$ and its associated dual variables $\bu_{ts}$. In addition, Lemma~\ref{lemma:ADMM3} shows that  \eqref{eq:ADMM3}-\eqref{eq:ADMM33} is equivalent to the update formula \eqref{eq:ADMM1}-\eqref{eq:ADMM13}. Its proof is deferred to Section~\ref{sec:proof}.
\begin{lemma}\label{lemma:ADMM3}
The update formula \eqref{eq:ADMM3}-\eqref{eq:ADMM33} with $\rho=\rho_0$ is equivalent to the update formula \eqref{eq:ADMM1}-\eqref{eq:ADMM13} with $\rho=2\rho_0$.
\end{lemma}

\subsection{Implementation and its computational cost per iteration}\label{sec:implementation}
Based on the update formulas \eqref{eq:ADMM3}-\eqref{eq:ADMM33}, the implementation of our proposed algorithm is  described as Algorithm~\ref{alg:ADMM3}.

As this algorithm depends on the graph decomposition $\calE_0\cup\calE_1$, in practice, we use a greedy algorithm to find $\calE_0$ as follows: First, label all edges in some arbitrary order and $\calE_0$ be an empty set. Second, cycle once through each edge and add it to $\calE_0$ if it is not  neighboring any existing edges in $\calE_0$. In fact, $\calE_0$ is called matching in graph theory and there are numerous algorithms for finding a matching within a graph.

\begin{algorithm}
\caption{The implementation of the ADMM method in \eqref{eq:ADMM3}.}\label{alg:ADMM3}
{\bf Input:}  Graph $(\calV,\calE)$ and its partition $\calE=\calE_0\cup\calE_1$ ($\calE_0$ has not neighboring edges); loss functions $\{f_i\}_{i\in\calV}$; parameters $\rho$ and $\lambda$. \\
{\bf Initialization:}  Initialize $\{\bx^{(0)}_i\}, \{\bz^{(0)}_i\}_{i\in\calV}\subset\reals^p$, $\{\bu^{(0)}_{st}\}_{(s,t)\in\calE_1}\subset\reals^p$.\\
{\bf Loop:} Iterate Steps 1--4 until convergence:\\
{\bf \it{1}:} For any $(s,t)\in\calE_0$, $(\bx_s^{(k+1)},\bx_t^{(k+1)})=\argmin_{\bx_s,\bx_t}f_s(\bx_s)+f_t(\bx_t)+\lambda\|\bx_s-\bx_t\|+\rho(d_s\|\bx_s\|^2+d_t\|\bx_t\|^2)+\bx_s^T\bt_s^{(k)}+\bx_t^T\bt_t^{(k)}$, where \[\bt_i^{(k)}=\sum_{j:(i,j)\in\calE_1}[-(\bu_{ij}^{(k)}+\rho\bz_{ij}^{(k)})-\rho(\bx_i^{(k)}+\bx_j^{(k)})]+\sum_{j:(j,i)\in\calE_1}[(\bu_{ji}^{(k)}+\rho\bz_{ji}^{(k)})-\rho(\bx_i^{(k)}+\bx_j^{(k)})]\]
{\bf \it{2}:}For any $i\in\calV$ and does not belong to any edges in $\calE_0$, $\bx_i^{(k+1)}=\argmin_{\bx_i}f_i(\bx_i)+\rho d_i\|\bx_i\|^2+\bx_i^T\bt_i^{(k)}$.\\
{\bf \it{3}:} For any $(s,t)\in\calE_1$, $\bz_{st}^{(k+1)}=\argmin_{\bz_{st}}\frac{\rho}{2}\|\bz_{st}\|^2+\bz_{st}(\bu_{st}^{(k)}-\rho\bx_s^{(k+1)}+\rho\bx_t^{(k+1)})+\lambda\|\bz_{st}\|=\mathrm{threshold}(\bx_s^{(k+1)}-\bx_t^{(k+1)}-\frac{\bu_{st}^{(k)}}{\rho},\frac{\lambda}{\rho})$.\\
{\bf \it{4}:} For any $(s,t)\in\calE_1$,  $\bu_{st}^{(k+1)}=\bu_{st}^{(k)}+\rho(\bz_{st}^{(k+1)}-\bx_s^{(k+1)}+\bx_t^{(k+1)})$. \\
{\bf Output:} The solution to \eqref{eq:problem0}, $\hat{\bx}_i=\lim_{k\rightarrow\infty}\bx_i^{(k)}$ for all $i\in\calV$.
\vspace{0.3cm}
\end{algorithm}

To compare the computational cost per iteration Algorithm~\ref{alg:ADMM3} and the network lasso, we investigate a commonly used special case that $f_i(\bx_i)=\|\bx_i-\by_i\|^2$. Then step 1 in Algorithm~\ref{alg:ADMM3} requires the following Lemma~\ref{lemma:step1}. We skip its proof since it is relatively straightforward and has been discussed in works such as \citep{Hallac2015}.
\begin{lemma}\label{lemma:step1}For any $\ba,\bb\in\reals^p$,
\begin{align*}&\argmin_{\bx,\by\in\reals^p}c_1\|\bx-\ba\|^2+c_2\|\by-\bb\|^2+\lambda\|\bx-\by\|\\=&\begin{cases}(\frac{c_1\ba+c_2\bb}{c_1+c_2},\frac{c_1\ba+c_2\bb}{c_1+c_2}),\,\,\text{if $2c_1c_2\|\ba-\bb\|\leq (c_1+c_2)\lambda$}\\(\ba-\frac{\lambda}{2c_1}\frac{\ba-\bb}{\|\ba-\bb\|},\bb-\frac{\lambda}{2c_2}\frac{\bb-\ba}{\|\bb-\ba\|}),\,\,\text{otherwise}.\end{cases}\end{align*}
\end{lemma}
%
%

Now let us investigate the computational complexity per iteration of Algorithm~\ref{alg:ADMM3}, by keep track of the multiplications of a scalar and a vector of $\reals^p$ (denoted as multiplications) and the additions of two vectors of $\reals^p$ (denoted as additions). In particular, the calculation of $\bt_s/(1+\rho_s)$ in the update of $x$ requires $2|\calE_1|+n$ multiplications and $2|\calE_1|+n$ additions between, and  step 1 requires an additional cost of at most $2|\calE_0|$ multiplications and $2|\calE_0|$ additions, and $|\calE_0|$ operations of finding the norm of a vector of length $p$ and $|\calE_0|$ operations of comparing two scalars. Step 3 requires $3|\calE_1|$ additions, $|\calE_1|$ multiplications and $|\calE_1|$ comparisons. Step 4 requires $3|\calE_1|$ additions and $|\calE_1|$ multiplications.

Note that the network lasso algorithm is equivalent to the case where $\calE_0=\emptyset$ and $\calE_1=\calE$, we may compare the computational cost per iteration between Algorithm~\ref{alg:ADMM3} and the network lasso algorithm, and it is clear that Algorithms~\ref{alg:ADMM3} has a smaller computational cost per iteration compared to network lasso.

\subsection{Convergence Rate}\label{sec:convergence}
Now let us investigate the theoretical convergence rate. First, we introduce a general theory on the local convergence of ADMM. Its proof is deferred to Section~\ref{sec:proof}.
\begin{thm}[Local Convergence Rate of ADMM]\label{thm:convergence}
Considering the problem of minimizing $f_1(\bx_1)+f_2(\bx_2)$ subject to $\bA_1\bx_1+\bA_2\bx_2=\bb$. Assuming that around the solution $(\bx^*,\by^*)$, $\partial f_1(\bx) = \bC_1\bx+\bc_1$ and $\partial f_2(\bx) = \bC_2\bx+\bc_2$, then the local convergence rate of the ADMM algorithm is  $O(c(\rho)^k)$, where $k$ is the number of iterations and $c(\rho)$ is the largest real components among all eigenvalue of
\[
\frac{1}{2}[(\bI-2(\bI+\rho \bA_2\bC_2^{-1}\bA_2^T)^{-1})(\bI-2(\bI+\rho \bA_1\bC_1^{-1}\bA_1^T)^{-1})+\bI].
\]
\end{thm}
Considering that Algorithm~\ref{alg:ADMM3} is obtained from solving \eqref{eq:proposed1}, the convergence rate of  Algorithm~\ref{alg:ADMM3} follows from this theorem with
\[
f_1(\{\bx_i\}_{i\in\calV})=\sum_{i\in\calV}f_i(\bx_i)+\lambda\sum_{(s,t)\in\calE_0}\|\bx_s-\bx_t\|,\,\,f_2(\{\bz_{st}\}_{(s,t)\in\calE_1})=\lambda\sum_{(s,t)\in\calE_1}\|\bz_{st}-\bz_{ts}\|.
\]
That is, $\bx_1$ in Theorem~\ref{thm:convergence} is replaced by $\{\bx_i\}_{i\in\calG}$,  $\bx_2$ in Theorem~\ref{thm:convergence} is replaced by $\{\bz_{st}\}_{(s,t)\in\calE_1}$, and $\bA_1\bx_1+\bA_2\bx_2=\bb$ is replaced by $\bx_s=\bz_{st}$ and $\bx_t=\bz_{ts}$ for all $(s,t)\in\calE_1$.
Therefore, we have $\bA_1\in\reals^{np\times 2p|\calE_1|}$, defined such that $\bA_1(2i-1,s_i)=\bI_{p\times p}$ and $\bA_1(2i,t_i)=\bI_{p\times p}$ if $(s_i,t_i)$ is the $i$-th edge in $\calE_1$, and $\bA_2=-\bI_{2p|\calE_1|\times 2p|\calE_1|}$. The matrix $\bC_1\in\reals^{pn\times pn}$ can be generated by the following three steps. First, the $(i,i)$-th $p\times p$ block is given by
\[
\bC_1(i,i)=\text{Hessian} f_i(\bx_i^*)
\]
Second, for $(i,j)\in\calE_0$ we let  $\bT(i,j)=\frac{1}{\|\bx_i^*-\bx_j^*\|}\bI-\frac{1}{\|\bx_i^*-\bx_j^*\|^3}(\bx_i^*-\bx_j^*)(\bx_i^*-\bx_j^*)^T$ if  $ \bx_i^*\neq \bx_j^*$, and $\bT(i,j)=\infty\bI$ if $\bx_i^*= \bx_j^*$. Third, we update the $(i,i), (i,j), (j,i), (j,j)$-th $p\times p$ blocks of $\bC_1$ by
\begin{align*}
\bC_1(i,i)\leftarrow \bC_1(i,i)+\bT(i,j), \,\,\,\,\,\,\bC_1(j,j)\leftarrow \bC_1(j,j)+\bT(i,j), \\\bC_1(i,j)\leftarrow \bC_1(i,j)-\bT(i,j), \,\,\,\,\,\,\bC_1(j,i)\leftarrow \bC_1(j,i)-\bT(i,j).\end{align*} 

The matrix $\bC_2\in\reals^{2p|\calE_1|\times 2p|\calE_1|}$ is generated as follows: if the $i$-th edge in $\calE_1$, $(s_i,t_i)$, satisfies that $\bx_{s_i}^*\neq \bx_{t_i}^*$, then the $(i,i)$-th $2p\times 2p$ block of $\bC_2$ is given by $[\bT(s_i,t_i),-\bT(s_i,t_i);-\bT(s_i,t_i),\bT(s_i,t_i)]$. The remaining $2p\times 2p$ blocks are all zero matrices.

Note that the network lasso algorithm is equivalent to Algorithm~\ref{alg:ADMM3} with $\calE_0=\emptyset$ and $\calE_1=\calE$, this result can also be applied to analyze the convergence rate of the network lasso algorithm.

While it is difficult to compare their convergence rates in general due to the complexities of $\bA_i$ and $\bC_i$, we can calculate the convergence rate numerically for some specific examples. Here we assume that $\calG$ is the one-dimensional chain graph defined by $\calV=\{1,2,\cdots,n\}$ and $\calE=\{(1,2),(2,3),\cdots,(n-1,n)\}$, and the partition such that $\calE_0=\{(1,2),(3,4),\cdots\}$ and $\calE_1=\{(2,3),(4,5),\cdots\}$.  We first compare the theoretical convergence rate  (measured by $c(\rho)$ in Theorem~\ref{thm:convergence}) of Algorithm~\ref{alg:ADMM3} and the network lasso $c(\rho)$  for the following three settings:
\begin{enumerate}
\item $\bx_i^*\in\reals^2$, $\hat{\bx}_i\neq \hat{\bx}_{i+1}$  when $i=50$, $\lambda=1$.
\item $\bx_i^*\in\reals^2$, $\hat{\bx}_i\neq \hat{\bx}_{i+1}$ when $i=10,20,\cdots, 90$, $\lambda=1$.
\item $\bx_i^*\in\reals^2$, $\hat{\bx}_i\neq \hat{\bx}_{i+1}$ when $i=2,4,\cdots, 98$, $\lambda=1$.
\end{enumerate}

The comparison of $c(\rho)$ of Algorithm~\ref{alg:ADMM3} and network lasso is visualized in Figure~\ref{fig:compare_local}. We can see that  Algorithm~\ref{alg:ADMM3} consistently has a smaller $c(\rho)$, which implies a faster convergence rate.

\begin{figure}
\centering
\includegraphics[width=0.32\textwidth]{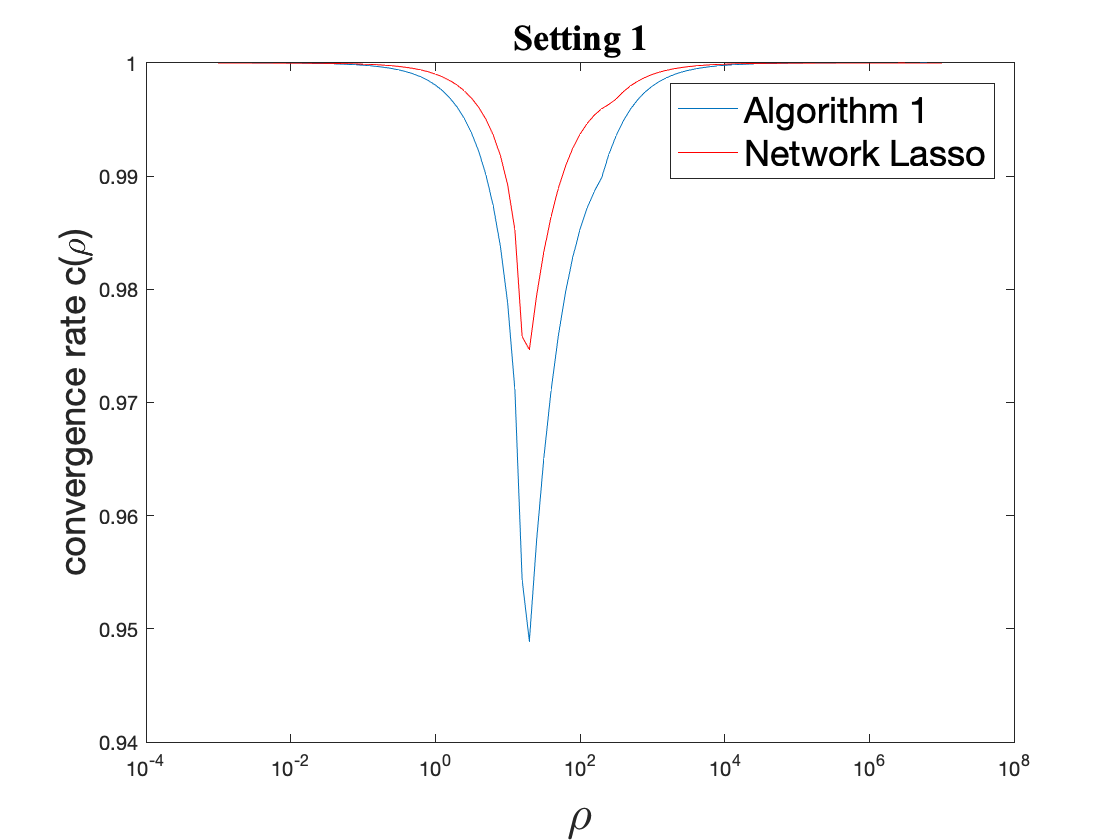}
\includegraphics[width=0.32\textwidth]{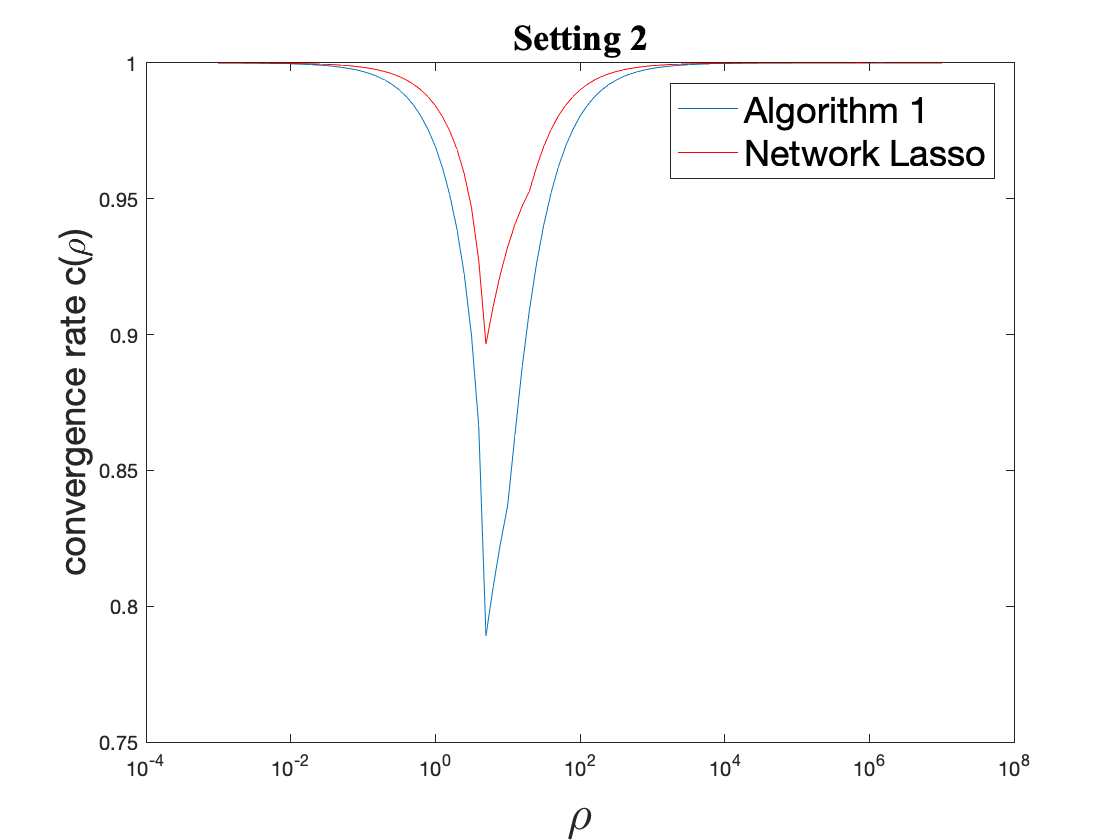}
\includegraphics[width=0.32\textwidth]{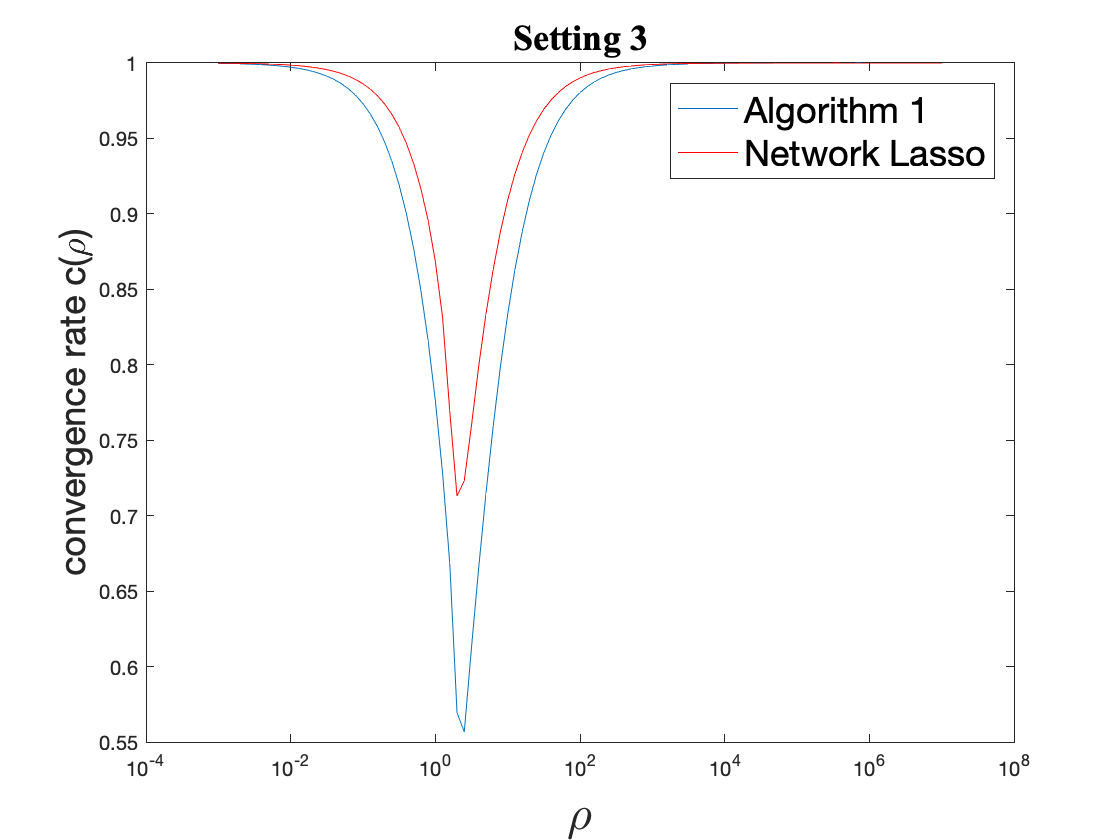}
\caption{Comparison of the theoretical local convergence rates between Algorithm~\ref{alg:ADMM3}, the network lasso, and the standard lasso.}\label{fig:compare_local}\end{figure}

\subsection{Comparison with other works based on graph decomposition}
 Decomposing a graph into edges or paths is an idea that has been applied in existing works. However, we remark that our approach is different from previous works. For example, \citep{Tansey2015} investigates the idea that decomposes the graph into a set of trails (this idea is also explored by Barbero and Sra \citep{barberoTV14} for the two-dimensional grid graph), and then apply existing algorithms to solve each problem. In particular, it decomposes $\calE$ into $K$ sets $\calE_1\cup\cdots\cup\calE_K$ such that for each $1\leq k\leq K$, $\calE_k$  is a trail. By writing the optimization problem as
 \begin{align*}
 \min_{\bx,\bz}\sum_{i\in\calV}f_i(\bx_i)+\lambda\sum_{1\leq k\leq K}\sum_{(s,t)\in\calE_k}\|\bz_{\calE_k,s}-\bz_{\calE_k,t}\|,\\
\text{s.t. $\bz_{\calE_k,s}=\bx_s$ for all $1\leq k\leq K$ and $s$ in some edge of $\calE_k$,}
 \end{align*}
then the ADMM algorithm can be used to update $\bx$ and $\bz$ alternatively. We remark that there are two main differences: first, their method only works well for the case where $\bx_i$ are scalars (i.e., $p=1$). In comparison, our method can handle the case where $\bx_i\in\reals^p$ with $p>1$. Second and more importantly, the total variation penalty term in their method was not partitioned and it is addressed using the augmented variable $\bz$; while in our case, the total variation penalty term is partitioned and part of the $\ell_1$ penalty was handled through the variable $\bx$. In fact, the idea in  this work can be combined with their idea for the case $p=1$ and the problem could be written as follows: first, we decompose $\calE$ into sets $(\calE^1_1\cup\cdots\cup\calE^1_{K_1})\cup(\calE^0_1\cup\cdots\cup\calE^0_{K_0})$, such that all  $\calE_k^0$ and $\calE_k^1$ are trails, and the trails $\calE^0_1, \cdots, \calE^0_{K_0}$ are disjoint. Then writing the optimization problem as
\begin{align*}
 \min_{\bx,\bz}\left(\sum_{i\in\calV}f_i(\bx_i)+\lambda\sum_{1\leq k\leq K_0}\sum_{(s,t)\in\calE^0_k}\|\bx_s-\bx_t\|\right)+\lambda\sum_{1\leq k\leq K_1}\sum_{(s,t)\in\calE^1_k}\|\bz_{\calE^1_k,s}-\bz_{\calE^1_k,t}\|,\\
\text{s.t. $\bz_{\calE^1_k,s}=\bx_s$ for all $1\leq k\leq K-1$ and $s$ in some edge of $\calE^1_k$.}
 \end{align*}
This formulation would give another algorithm for solving the problem in \citep{Tansey2015}, but we will leave it for possible future investigations.

\section{Experiments}
In this section,  the proposed Algorithm~\ref{alg:ADMM3} will be compared with network lasso for solving \eqref{eq:problem0} under various scenarios. We measure their error at iteration $k$ by the difference between its objective value and the optimal objective value.
We remark that all ADMM algorithms require an augmented Lagrangian parameter $\rho$, and the algorithms would converge slowly when $\rho$ is too large or too small.  While there have been many works on the choice of $\rho$. For example, a simple varying penalty strategy based on residual balancing is suggested in Section 3.4.1 of \citep{Boyd:2011:DOS:2185815.2185816}, and another choice based on the Barzilai-Borwein spectral method for gradient descent is proposed in there do not exists an optimal choice for settings~\citep{Xu2017AdaptiveAW}. Considering that there is no general consensus on the optimal strategy of the choice of $\rho$, we will test the performance of the algorithms on a range of $\rho$ in the following simulations, and we choose the range so that the optimal $\rho$ is inside the chosen range.

\subsection{Simulations}
We first test our result on the one-dimensional chain graph. Following \citep{Zhu2015}, we use the model that 
\[
\by^*_i=\begin{cases}[1,1],\text{if $1\leq i\leq 11$} \\ [-1,1],\text{if $12\leq i\leq 22$}\\ [2,2],\text{if $1\leq 23\leq 33$}\\ [-1,-1],\text{if $34\leq i\leq 44$}\\ [0,0],\text{if $i\geq 45$}\end{cases}.
\] 
Then we let $\by_i=\by_i^*+N(\mathbf{0},\bI_{2\times 2})$ and $\lambda =1$ or $10$, and compare Algorithm~\ref{alg:ADMM3} will be compared with network lasso  in Figure~\ref{fig:compare1}  with various choices of $\rho$. The figures indicate that for both choices of $\lambda$, Algorithm~\ref{alg:ADMM3} always performs better with a good choice of $\rho$. In fact, if $\rho$ is chosen to be the optimal values for both algorithms, Algorithm~\ref{alg:ADMM3} converges twice as fast as network lasso.

\begin{figure}
\centering
\includegraphics[width=0.32\textwidth]{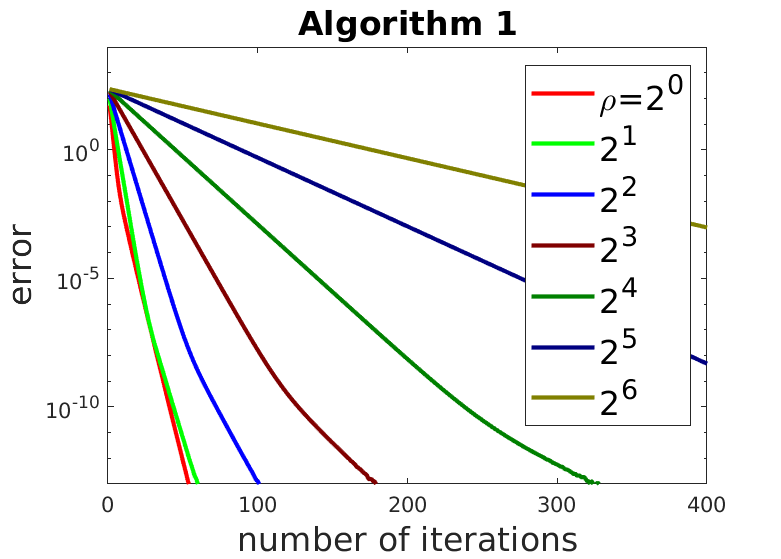}
\includegraphics[width=0.32\textwidth]{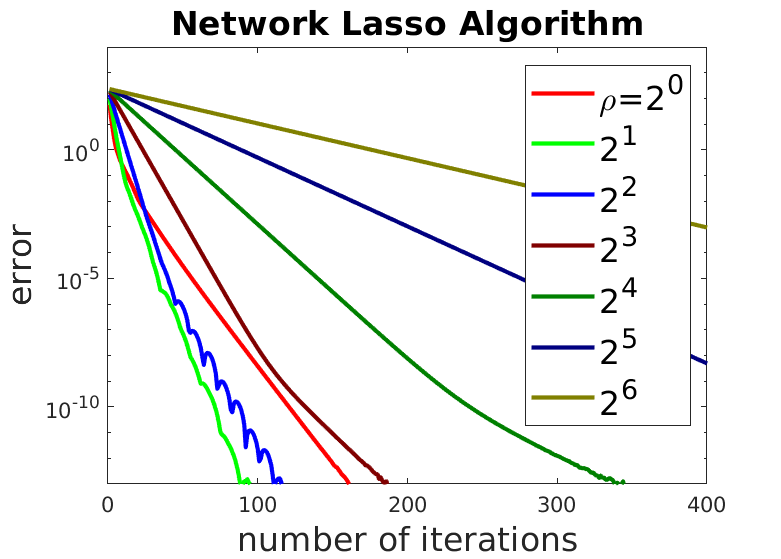}\\
\includegraphics[width=0.32\textwidth]{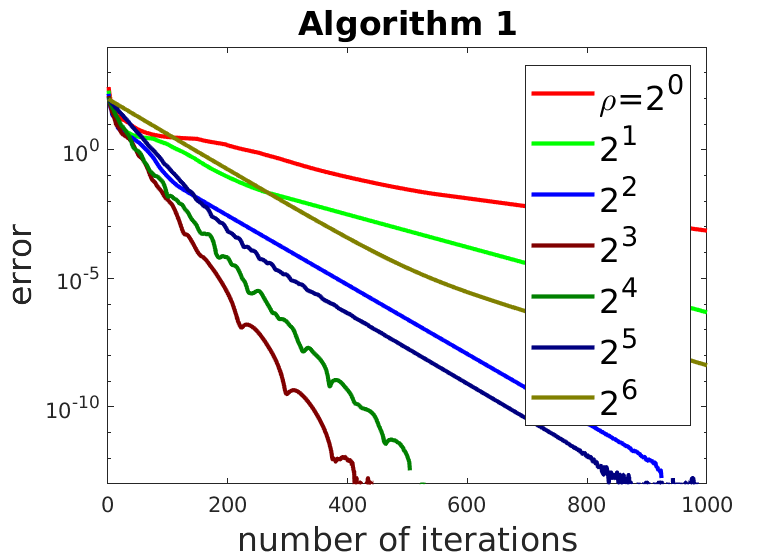}
\includegraphics[width=0.32\textwidth]{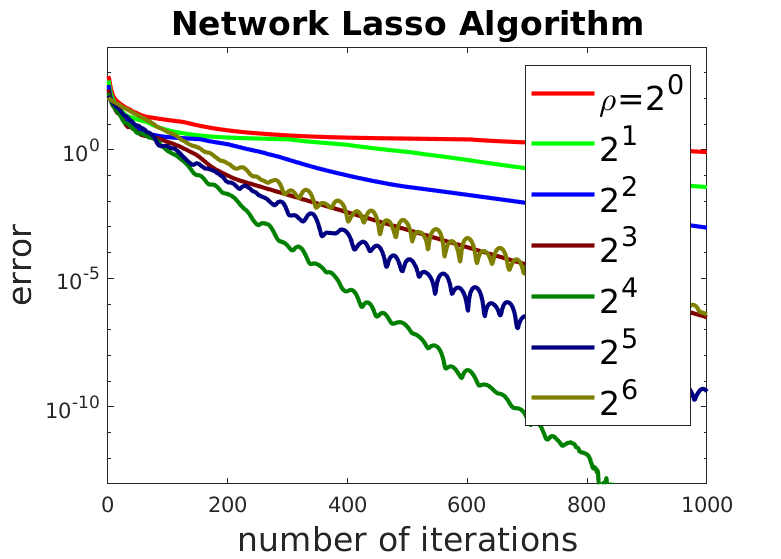}
\caption{Comparison the convergence rates under the 1D chain graph setting with $\lambda=1$ (top row) and $\lambda=10$ (second row).}\label{fig:compare1}\end{figure}
\begin{figure}
\centering
\includegraphics[width=0.32\textwidth]{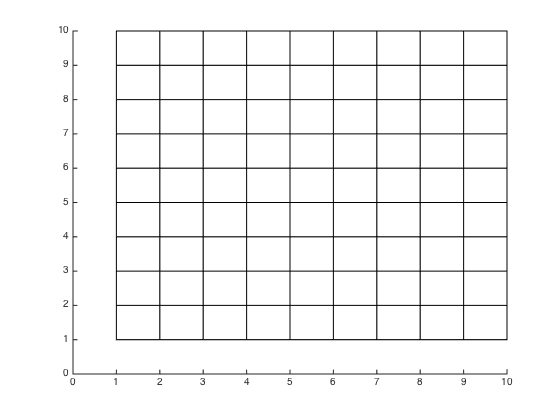}
\includegraphics[width=0.32\textwidth]{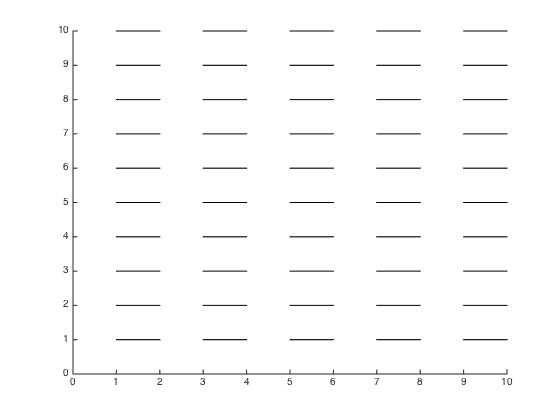}
\caption{Visualization of a two-dimensional grid graph of size $10\times 10$ (left) and the corresponding $\calE_0$ for this graph (right).}\label{fig:grid}
\end{figure}

\begin{figure}
\centering
\includegraphics[width=0.32\textwidth]{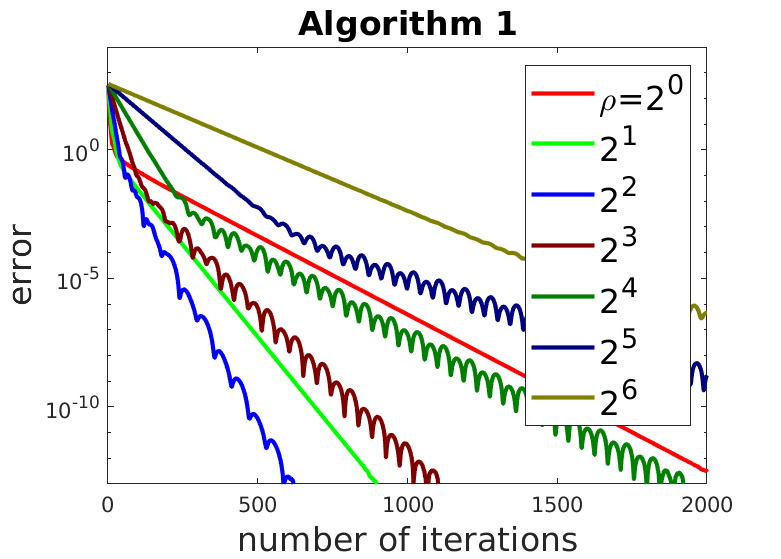}
\includegraphics[width=0.32\textwidth]{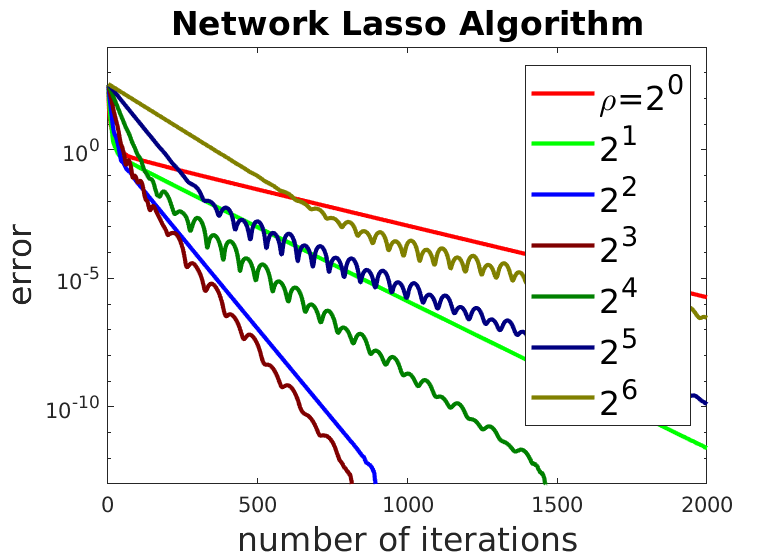}\\
\includegraphics[width=0.32\textwidth]{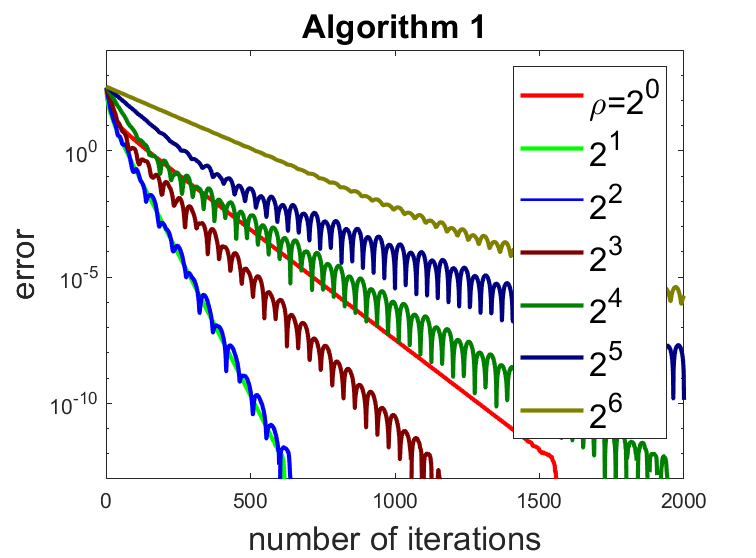}
\includegraphics[width=0.32\textwidth]{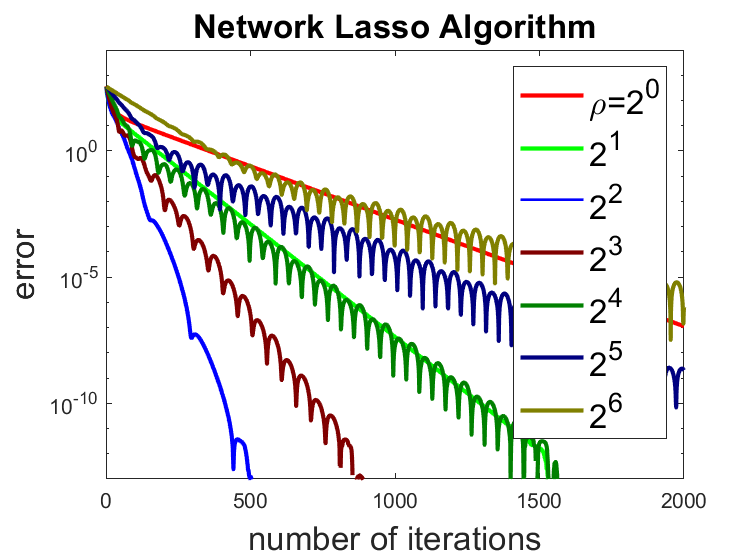}
\caption{Comparison the convergence rates under the 2D grid graph setting with $\lambda=1$ (top row) and $\lambda=5$ (second row).}\label{fig:compare3}
\end{figure}
\begin{figure}
\centering
\includegraphics[width=0.32\textwidth]{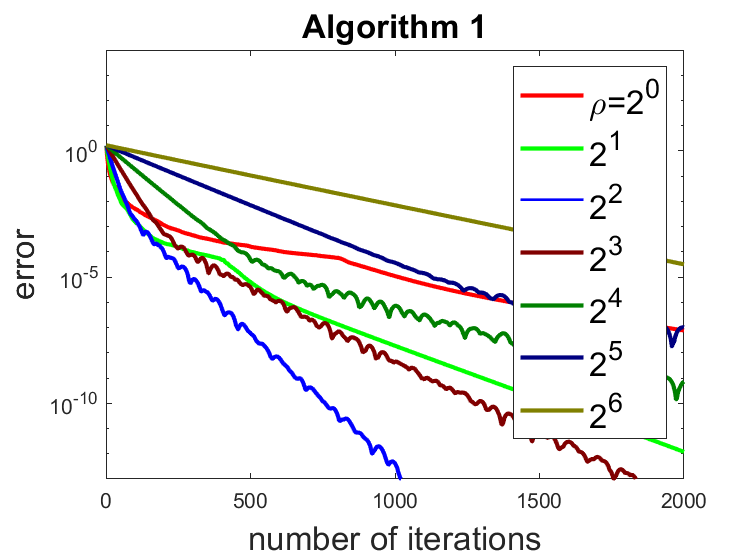}
\includegraphics[width=0.32\textwidth]{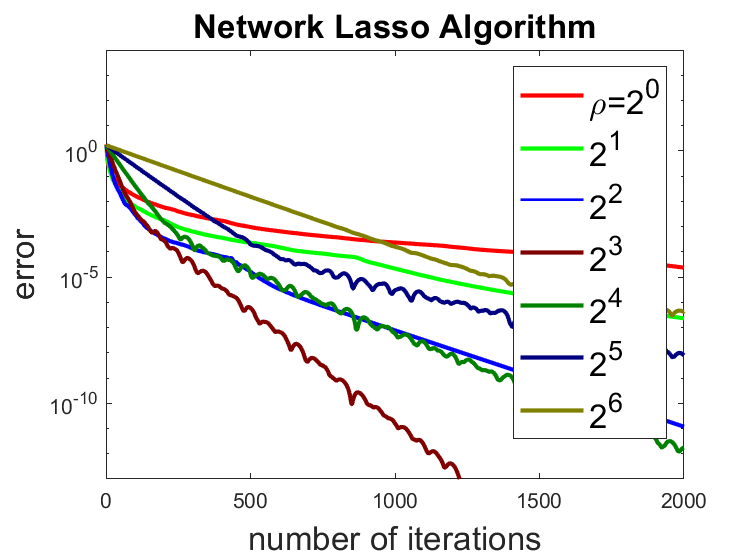}\\
\includegraphics[width=0.32\textwidth]{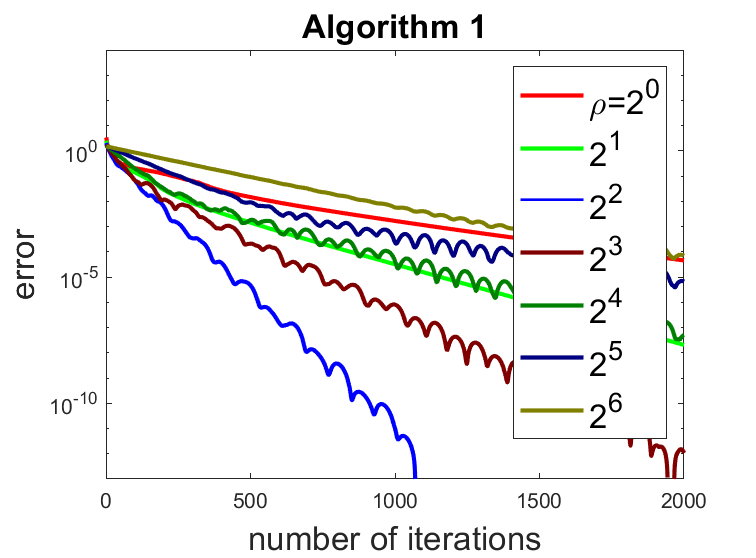}
\includegraphics[width=0.32\textwidth]{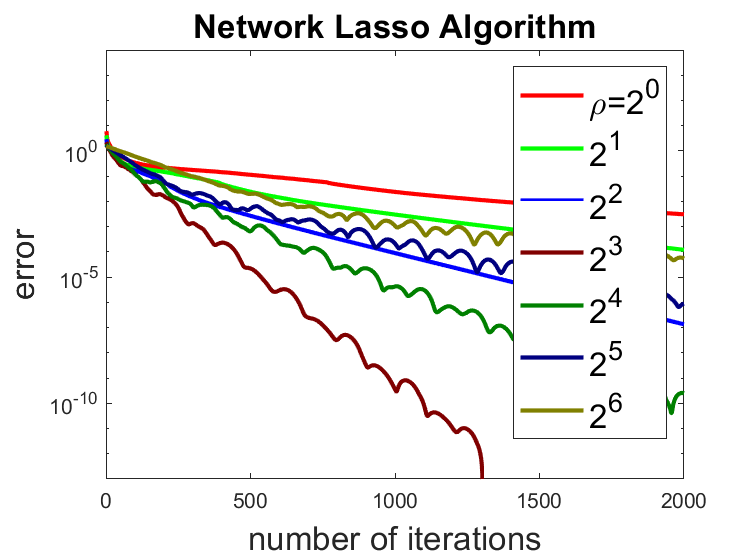}
\caption{Comparison the convergence rates for the Chicago crime dataset with $\lambda=0.05$ (top row) and $\lambda=0.25$ (second row).}\label{fig:compare4}
\end{figure}
We also test the two-dimensional grid graph example. We generate data in the form of a $64$ by $64$ grid of points; the value is equal to $[0,0,0]$ for points within a distance of $16$ from the middle of the grid, and $[0.4,0.7,1]$ for all other points, and add noise of $N(\mathbf{0},\bI_{3\times 3})$ for all points. For this example, we have a natural choice of $\calE_0$ as visualized in the right Figure of Figure~\ref{fig:grid}.  The convergence rates are shown in Figure~\ref{fig:compare3}, which shows that Algorithm~\ref{alg:ADMM3} has a comparable or faster convergence rate as the network lasso algorithm ADMM. Combining it with the fact that Algorithm~\ref{alg:ADMM3} has a smaller computational complexity per iteration, this implies the numerical superiority of Algorithm~\ref{alg:ADMM3}.

\subsection{Real dataset}
We also use an example of a graphical fused lasso problem with a reasonably large, geographically-defined underlying graph. The data comes from police reports made publicly available by the city of Chicago, from 2001 until the present (Chicago Police Department 2014) and is available as the supplementary files of~\citep{Arnold2016}. In this dataset, the vertices represent the census blocks and the edges represent the neighboring blocks, and more detailed explanation of this dataset is deferred to the supplementary file. This graph has $2162$ vertices, $6995$ edges and by running the greedy algorithm as described in Section 2.2, we obtain $\calE_0$ with $1025$ edges and $\tilde{m}_0=797$.  The result of the experiment in Figure~\ref{fig:compare4}  shows when $\lambda=0.05$ or $0.25$, Algorithm~\ref{alg:ADMM3} converges much faster than the network lasso algorithm. Since Algorithm~\ref{alg:ADMM3} has a smaller computational complexity per iteration as analyzed, this experiment shows that our algorithm numerically converges faster than network lasso.
\section{Conclusion}

This paper proposes a new ADMM algorithm for solving graphic fused lasso, based on a novel method of dividing the objective function into two components. Compared with the standard ADMM algorithm of network lasso, it has a similar complexity per iteration while usually converges within fewer iterations. As for future directions, it would be interesting to theoretically analyze its advantage, and  explore other ways of dividing the objective function which could even future improve the performance of the ADMM algorithm for graph-fused lasso.

The idea of the proposed algorithm can also be applied to other problems such as trend filtering~\citep{Ramdas2015}, defined by $
\argmin_{\bx_i, i=1,\cdots,n}\sum_{i=1}^nf_i(\bx_i)+\lambda\sum_{i=2}^{n-1}\|\bx_{i-1}-2\bx_i+\bx_{i+1}\|.$
We may divide the objective function into $\sum_{i=1}^nf_i(\bx_i)+\|\bx_1-2\bx_2+\bx_3\|+\|\bx_4-2\bx_5+\bx_6\|+\cdots$ and $\|\bx_2-2\bx_3+\bx_4\|+\|\bx_3-2\bx_4+\bx_5\|+\|\bx_5-2\bx_6+\bx_7\|+\|\bx_6-2\bx_7+\bx_8\|+\cdots$. The analysis of its performance and the comparison with standard algorithms would be another possible future direction.

\section{Proofs}\label{sec:proof}

\subsection{Proof of Lemma~\ref{lemma:ADMM3}}
Let us consider the problem
\begin{align}\label{eq:proposed2}
\argmin_{\{\bx_i\}_{i\in\calE}, \{\bz_{st}\}_{(s,t)\in\calE_1}} \left(\sum_{i\in\calE}f_i(\bx_i)+\lambda\sum_{(s,t)\in\calE_0}\|\bx_s-\bx_t\|\right) + \lambda\sum_{(s,t)\in\calE_1}\|\bz_{st}\|,\\\,\,\text{s.t. $\bz_{st}=\bx_s-\bx_t$, $\bz_{ts}=\bx_s+\bx_t$.}\nonumber
\end{align}
and its associated Lagrangian of
\begin{align}\label{eq:lagrangian2}
\bar{L}_\rho(x,z,u)=\sum_{i\in\calV}f_i(\bx_i)&+\lambda\sum_{(s,t)\in\calE_0}\|\bx_s-\bx_t\|+\sum_{(s,t)\in\calE_1}\Big(\lambda \|{\bz}_{st}\|+\bu_{st}^T(\bz_{st}-\bx_s+\bx_t)\\&+\bu_{ts}^T(\bz_{ts}-\bx_s-\bx_t)+\frac{\rho}{2}\|\bz_{st}-\bx_s+\bx_t\|^2+\frac{\rho}{2}\|\bz_{ts}-\bx_s-\bx_t\|^2\Big),
\end{align}
as well as the ADMM algorithm of
\begin{align}\label{eq:ADMM2}
&x^{(k+1)}=\argmin_{x}\bar{L}_\rho(x,z^{(k)},u^{(k)})\\
&z^{(k+1)}=\argmin_{z}\bar{L}_\rho(x^{(k+1)},z,u^{(k)})\\
&u_{st}^{(k+1)}=u_{st}^{(k)}+\rho(z_{st}^{(k+1)}-x_s^{(k+1)}+x_t^{(k+1)}),\,\,u_{ts}^{(k+1)}=u_{ts}^{(k)}+\rho(z_{ts}^{(k+1)}-x_s^{(k+1)}-x_t^{(k+1)}).\label{eq:ADMM23}
\end{align}
It can be verified that the update of \eqref{eq:ADMM2}-\eqref{eq:ADMM23} with $\bar{L}_\rho$ is in fact identical to the update of \eqref{eq:ADMM1}-\eqref{eq:ADMM13} with $\hat{L}_{2\rho}$. In addition, using the fact that \eqref{eq:proposed2} is  \eqref{eq:proposed3} with  additional constraints $\bz_{ts}=\bx_s+\bx_t$, and Lemma~\ref{lemma:augmented}
 implies the equivalence between \eqref{eq:ADMM2}-\eqref{eq:ADMM23} is identical to \eqref{eq:ADMM3}-\eqref{eq:ADMM33}.
\begin{lemma}\label{lemma:augmented}
The preconditioned ADMM procedure for solving $\min_{\bx,\by} f(\bx)+g(\by)$ subject to $\bA\bx+\bB\by=\bc$
\begin{align}
\bx^{(k+1)}&=\argmin_{\bx}L(\bx,\by^{(k)},\bv^{(k)})+\frac{\rho}{2}(\bx-\bx^{(k)})^T\bC_1^T\bC_1(\bx-\bx^{(k)}),\label{eq:precondition_ADMM1}\\ \by^{(k+1)}&=\argmin_{\by}L(\bx^{(k+1)},\by,\bv^{(k)})+\frac{\rho}{2}(\by-\by^{(k)})^T\bC_2^T\bC_2(\by-\by^{(k)}),\label{eq:precondition_ADMM2}\\ \bv^{(k+1)}&=\bv^{(k)}+\rho(\bA\bx^{(k+1)}+\bB\by^{(k+1)}-\bc),\label{eq:precondition_ADMM3}
\end{align}
where  $L(\bx,\by,\bv)=f(\bx)+g(\by)+\bv^T(\bA\bx+\bB\by-\bc)+\frac{\rho}{2}\|\bA\bx+\bB\by-\bc\|^2$, is equivalent to the standard ADMM procedure applied to the augmented problem
\[
\min_{\bx,\by}  f(\bx)+g(\by),\,\,\text{s.t. $[\bA\bx+\bB\by,\bC_1\bx,\bC_2\by]=[\bc,\bz,\bw]$.}
\]
\end{lemma}

\subsection{Proof of Lemma~\ref{lemma:augmented}}
\begin{proof}[Proof of Lemma~\ref{lemma:augmented}]
Applying the standard ADMM routine to optimize $(\bx,\bw)$ and $(\by,\bz)$ alternatively, the update formula for the augmented ADMM is
\begin{align}
\bx^{(k+1)}&=\argmin_{\bx}L(\bx,\by^{(k)},\bv^{(k)})+\frac{\rho}{2}\|\bC_1^{0.5}\bx-\bz^{(k)}\|^2+\bv_1^{(k)\,T}(\bC_1^{0.5}\bx-\bz^{(k)}),\label{eq:augmented_ADMM1}\\
\bw^{(k+1)}&=C_2^{0.5}\by^{(k)}+\frac{1}{\rho}\bv_2^{(k)},\label{eq:augmented_ADMM2}
\\\by^{(k+1)}&=\argmin_{\by}L(\bx^{(k+1)},\by,\bv^{(k)})+\frac{\rho}{2}\|\bC_2^{0.5}\by-\bw^{(k+1)}\|^2+\bv_2^{(k)\,T}(\bC_2^{0.5}\by-\bw^{(k+1)}),\label{eq:augmented_ADMM3}\\ \bv^{(k+1)}&=\bv^{(k)}+\rho(\bA\bx^{(k+1)}+\bB\by^{(k+1)}-\bc), \label{eq:augmented_ADMM4} \\
\bz^{(k+1)}&=C_1^{0.5}\bx^{(k+1)}+\frac{1}{\rho}\bv_1^{(k)}\\
\bv_1^{(k+1)}&=\bv_1^{(k)}+\rho(\bC_1^{0.5}\bx^{(k+1)}-\bz^{(k+1)}),\,\, \bv_2^{(k+1)}=\bv_2^{(k)}+\rho(\bC_2^{0.5}\by^{(k+1)}-\bw^{(k+1)})\label{eq:augmented_ADMM5}.
\end{align}
Note that by plugging the definition of $\bz^{(k+1)}$ in \eqref{eq:augmented_ADMM4} to the definition of $\bv^{(k+1)}$ \eqref{eq:augmented_ADMM5}, $\bv_1^{(k+1)}=0$ for all $k$. So \eqref{eq:augmented_ADMM4}  implies that $\bz^{(k+1)}=C_1^{0.5}\bx^{(k+1)}$ and \eqref{eq:augmented_ADMM1} is equivalent to \eqref{eq:precondition_ADMM1}.

Plugging in the definition of $\bw^{(k+1)}$ to \eqref{eq:augmented_ADMM3}, we obtain the equivalence between \eqref{eq:augmented_ADMM3} and  \eqref{eq:precondition_ADMM2}.
\end{proof}
\subsection{Proof of Theorem~\ref{thm:convergence}}
\begin{proof}[Proof of Theorem~\ref{thm:convergence}]
By calculation, the ADMM algorithm is equivalent to the Douglas-Rachford splitting method applied to
\[
\max_{\bz}-\bb^T\bz-f_1^*(-\bA_1^T\bz)-f_2^*(-\bA_2^T\bz)
\]
with two parts given by $f=f_2^*(-\bA_2^T\bz)$ and $g=\bb^T\bz+f_1^*(-\bA_1^T\bz)$ respectively, and the  Douglas-Rachford splitting method is an iterative method that minimizing $f(\bx)+g(\bx)$ with update formula
\[
\bx^{(k+1)}=\frac{1}{2}[(\bI-2\prox_{\rho f})(\bI-2\prox_{\rho g})+\bI](\bx^{(k)}).
\]

Note that locally around the optimal solution we have
\begin{align*}
&\prox_{\rho f}=(\bI+\rho \partial f)^{-1},\,\,\,\partial f(\bx)= \bA_2\bC_2^{-1}\bA_2^T\bx+\bc_1,\\&\prox_{\rho g}=(\bI+\rho \partial g)^{-1},\,\,\,\partial g(\bx)= \bA_1\bC_1^{-1}\bA_1^T\bx+\bc_2
\end{align*}
for some $\bc_1$ and $\bc_2$, each iteration of the algorithm is a linear operator in the form of 
\begin{align*}
&\frac{1}{2}[(\bI-2\prox_{\rho f})(\bI-2\prox_{\rho g})+\bI](\bx)
\\=&\frac{1}{2}[(\bI-2(\bI+\rho \bA_2\bC_2^{-1}\bA_2^T)^{-1})(\bI-2(\bI+\rho \bA_1\bC_1^{-1}\bA_1^T)^{-1})+\bI](\bx)+\bc_0.
\end{align*}
As a result, the algorithm converges in the order of
\[
\left(\frac{1}{2}[(\bI-2(\bI+\rho \bA_2\bC_2^{-1}\bA_2^T)^{-1})(\bI-2(\bI+\rho \bA_1\bC_1^{-1}\bA_1^T)^{-1})+\bI]\right)^k,
\]
where $k$ is the number of iterations. The theorem is then proved.
\end{proof}



\bibliographystyle{Chicago}
\bibliography{bib-online}
\end{document}